\documentclass[3p,times]{elsarticle}
\usepackage{bbm}
\usepackage{amsfonts}
\usepackage{pifont}
\usepackage{dashrule}
\usepackage[all,pdf]{xy}
\usepackage{mathrsfs}
\usepackage{amsmath}
\usepackage{amssymb}
\usepackage[amsmath, thmmarks]{ntheorem}
\usepackage{graphicx}
\usepackage{float}
\usepackage{tikz}
\usepackage{hyperref}
\hypersetup{
  hidelinks,
  pdftitle={A complete representation theorem for nullnorms on bounded trellises},
  pdfauthor={Zhenyu Xiu},
  pdfsubject={Representation of nullnorms on bounded trellises},
  pdfkeywords={bounded trellis, nullnorm, mixed interaction function, representation theorem}
}

\journal{Fuzzy Sets and Systems}
\biboptions{sort&compress}

\theorembodyfont{\normalfont}
\newtheorem{definition}{\bfseries Definition}[section]
\theorembodyfont{\itshape}
\newtheorem{lemma}{\bfseries Lemma}[section]
\newtheorem{theorem}{\bfseries Theorem}[section]
\theorembodyfont{\normalfont}
\newtheorem{remark}{\bfseries Remark}[section]
\theorembodyfont{\itshape}

\newtheorem{proposition}{\bfseries Proposition}[section]
\theorembodyfont{\normalfont}
\newtheorem{example}{\bfseries Example}[section]

\newcommand{\M}{\mathbf{M}}
\newcommand{\N}{\mathbf{N}}
\newcommand{\G}{\mathbf{G}}

\begin{document}
\begin{frontmatter}

\title{A complete representation theorem for  nullnorms on bounded
trellises\tnoteref{funding}}

\tnotetext[funding]{Supported by the National Natural Science Foundation
of China (Nos. 12671553 and 12271036).}
\author[gxmu]{Zhenyu Xiu}
\ead{xyz198202@163.com}
\address[gxmu]{School of Mathematical Science, Guangxi Minzu University,
Nanning 530006, Guangxi, China}

\begin{abstract}
We establish necessary and sufficient conditions under which a binary
operation on a bounded trellis is a proper nullnorm.  The
representation combines a t-conorm on the lower interval, a t-norm on the
upper interval, two order-preserving maps, and a commutative, increasing
function on $I_a^3\times I_a^3$, where $I_a^3$ consists of the elements
incomparable with the absorbing element $a$ that neither reach $a$ nor are
reachable from $a$. Unlike earlier range-restricted constructions, this
function may take values anywhere in the trellis. To preserve associativity
for such unrestricted values, we introduce a mixed interaction function that
evaluates every pair with at least one component in $I_a^3$.  We also derive the
specializations in which this region is empty or consists of a single
element, including an exact description of the admissible value in the
singleton case. A five-element lattice example shows that the mixed
associativity condition is independent of the remaining hypotheses, while
a fourteen-element nontransitive trellis example
 demonstrates the necessity of allowing the unrestricted range. Finally, the principal range-restricted
subclasses and the bounded-lattice case are recovered as specializations of
the general representation.

\end{abstract}


\end{frontmatter}

\section{Introduction}\label{sec:intro}


As generalizations of triangular norms (t-norms, for short) and triangular
conorms (t-conorms, for short), nullnorms on the real unit interval were
introduced by Calvo,
De Baets and Fodor~\cite{Calvo2001}, while the closely related
t-operators were introduced by Mas, Mayor and Torrens~\cite{Mas1999}.
Mas, Mayor and Torrens~\cite{Mas2002} subsequently showed that the two
notions have the same block structure and are equivalent. Idempotency
and distributivity were investigated, for example,
in~\cite{Drygas2004,Qin2005,Drewniak2008,Drygas2015}, while general
treatments of associative aggregation can be found
in~\cite{Alsina2006,Grabisch2009}.

Kara\c{c}al, \.{I}nce and Mesiar~\cite{Karacal2015} extended nullnorms
to bounded lattices. Subsequent construction and classification methods
include~\cite{Cayli2018,Ertugrul2018,Cayli2020,Hua2022,Wu2022,
ZhangBeam2022}. Unlike a chain, a lattice can contain elements
incomparable with the absorbing element, so the behavior on these
elements is a genuine part of a representation problem.

Sun and Liu~\cite{Sun2020} obtained a representation of the class
whose values are comparable with the absorbing element. Zhang,
Ouyang, Wang and De Baets~\cite{Zhang2022} subsequently established
a complete representation on bounded lattices, using  two order-preserving maps, a triangular conorm, a triangular norm and a conditionally associative
function.  Wang, Ouyang, Zhang and De Baets~\cite{Wang2023} examined the
independence of conditions in that representation. These results
suggest separating the part forced by the boundary operations from
the part on which associativity still needs to be imposed.

Trellises replace the partial order of a lattice by a reflexive and
antisymmetric relation that need not be transitive. Classical sources
include~\cite{Fried1970,Skala1971,Gladstien1973,Skala1972}. The loss of
transitivity is mathematically substantive rather than merely formal:
cycle-transitivity of reciprocal relations is studied
in~\cite{DeBaets2006,DeBaets2015}, while intransitive indifference is
surveyed in~\cite{Fishburn1970}. For aggregation, the meet and join of
a trellis cannot, in general, serve as a t-norm and a t-conorm,
respectively. Zedam
and De Baets~\cite{Zedam2023} developed triangular norms on bounded
trellises, and Kong and Zhao~\cite{Kong2024} studied uninorms in the
same setting.
 For nullnorms, Xiu and Zheng~\cite{Xiu2025} established structural
properties and a representation of the class $\mathcal V_a$, consisting
of nullnorms whose ranges are disjoint from $I_a^1\cup I_a^2$.
Jiang, Wang and Liu~\cite{Jiang2025}
provided additional construction methods and investigated restrictions
on the absorbing element.

The representation theorem established by Xiu and
Zheng~\cite{Xiu2025} provides a natural starting point for the present
study. It characterizes the subclass $\mathcal{V}_a$ using t-norms,
t-conorms, two order-preserving maps, and a commutative and
increasing function $H$. A key assumption underlying this
representation is that the range of $H$ excludes values in
$I_a^1\cup I_a^2$. Although this restriction makes the resulting
decomposition tractable, it is not implied by the defining axioms of
nullnorms; rather, it is an additional condition defining
$\mathcal{V}_a$.

For a general proper nullnorm, however, the value associated with a
pair $(x,y)\in I_a^3\times I_a^3$ may lie in
$I_a^1\cup I_a^2$. Removing the range restriction therefore creates a
domain mismatch: if $H(x,y)\notin I_a^3$, then the iterated expression
$H(H(x,y),z)$ is undefined because $H$ is defined only on
$I_a^3\times I_a^3$, even though the corresponding associativity
condition for the global operation must still be satisfied.

To address this mismatch, we allow
$
H:I_a^3\times I_a^3\rightarrow X
$
to take any value permitted by the stated compatibility conditions and
introduce a mixed interaction function $\Lambda$ on
$
(X\times I_a^3)\cup(I_a^3\times X).
$
The function $\Lambda$ evaluates each intermediate value produced by
$H$ through the appropriate branch of the global construction, thereby
making the iterated expressions required by associativity well defined.
The resulting construction theorem and its converse together yield a
complete representation of proper nullnorms with absorbing element $a$
and recover the representation of Xiu and Zheng~\cite{Xiu2025} as the
special case obtained by imposing the corresponding range restriction.

The remainder of the paper is organized as follows.
Section~\ref{sec:prelim} introduces the regional decomposition and the
preliminary results used later. Section~\ref{sec:main} establishes the
construction and complete representation theorems and presents two
illustrative examples. Section~\ref{sec:specializations} treats the
empty and singleton cases of $I_a^3$, discusses the range-restricted
subclasses, and considers the bounded-lattice specialization.
Section~\ref{sec:conclusion} concludes the paper.

\section{Preliminaries}\label{sec:prelim}

A \emph{pseudo-order} $\unlhd$ on a set $X$ is a reflexive and
antisymmetric binary relation. Write $x\lhd y$ if $x\unlhd y$
and $x\ne y$, and $x\parallel y$ if neither comparison holds. A set
equipped with a pseudo-order is called a \emph{pseudo-ordered set}
(\emph{psoset}, for short) and is denoted by $\mathbb P=(X,\unlhd)$.
The set
of all elements of $X$ that are incomparable with $a$ is denoted by
$I_a$; that is, 
$
I_a=\{x\in X\mid x\parallel a\}.
$ 
Let $\mathcal C\subseteq X$, let $x,y\in X$, and let $z,t\in\mathcal C$.
The relation $x\lesssim y$ means that there is a finite sequence
$(x_1,\ldots,x_n)$ in $X$ such that
$x\unlhd x_1\unlhd\cdots\unlhd x_n\unlhd y$. Similarly,
$z\lesssim_{\mathcal C}t$ means that such a sequence can be chosen in
$\mathcal C$. The set $\mathcal C$ is called a \emph{cycle} if both
$z\lesssim_{\mathcal C}t$ and $t\lesssim_{\mathcal C}z$ hold for all
$z,t\in\mathcal C$. By antisymmetry, every nontrivial cycle contains at
least three elements.
A \emph{trellis} is a pseudo-ordered set in which every pair has a
greatest lower bound and a least upper bound, denoted by $x\wedge y$
and $x\vee y$, respectively. It is \emph{bounded} if it has elements
$0$ and $1$ with $0\unlhd x\unlhd 1$ for every $x\in X$.
An element $a\in X$ is called  \emph{middle-transitive} if
$x\unlhd a\unlhd y$ implies $x\unlhd y$ for all $x,y\in X$; the set
of such elements is denoted by $X^{mtr}$.
These conventions follow the classical trellis literature
\cite{Fried1970,Skala1971,Gladstien1973,Skala1972}.
Throughout, $\mathbb P=(X,\unlhd,\wedge,\vee,0,1)$ is a bounded
trellis and $[b,d]=\{x:b\unlhd x\unlhd d\}$ whenever $b\unlhd d$.
Open and half-open intervals have the corresponding strict endpoint
conditions. All interval relations are restrictions of $\unlhd$.
For $a\in X$, define
\[
\begin{aligned}
I_a^1&=\{x\in I_a\mid x\lesssim a\},\\
I_a^2&=\{x\in I_a\mid a\lesssim x\},\\
I_a^3&=\{x\in I_a\mid x\not\lesssim a\text{ and }a\not\lesssim x\},  \\
K&=\bigl\{a\in X^{mtr}\mid \text{there are no }y\in
\mathopen{]}0,a\mathclose{[}\cup I_a^1\text{ and }z\in
\mathopen{]}a,1\mathclose{[}\cup I_a^2\text{ such that }z\unlhd y\bigr\}.
\end{aligned}
\]
These sets follow the notation of~\cite{Kong2024}.
For a function $F:D\to X$ and a subset $A\subseteq D$, write
$F(A)=\{F(u)\mid u\in A\}$.
For pairs, we use the componentwise relation
\[
(x,z)\preccurlyeq(y,t)\quad\Leftrightarrow \quad
x\unlhd y\quad\text{and}\quad z\unlhd t.
\]


A subset $A\subseteq X$ is called a \emph{lower set} if
$x\unlhd y$ and $y\in A$ imply $x\in A$. Dually, $A$ is called an
\emph{upper set} if $x\in A$ and $x\unlhd y$ imply $y\in A$.
 A map $q:D\to X$, where $D\subseteq X$, is called
\emph{order-preserving} if $x\unlhd y$ implies
$q(x)\unlhd q(y)$ for all $x,y\in D$.
Our terminology and notation for order-theoretic and lattice-theoretic
concepts follow~\cite{Gierz2003}.
 A subset
$\mathcal D\subseteq X^2$ is called \emph{symmetric} if
$(x,y)\in\mathcal D$ implies $(y,x)\in\mathcal D$ \cite{Zhang2022}.
Let $\mathcal D\subseteq X^2$ be symmetric.
A function
$F:\mathcal D\to X$ is called \emph{commutative} if
$F(x,y)=F(y,x)$ for every $(x,y)\in\mathcal D$, and it is called
\emph{increasing} if
$
F(x,z)\unlhd F(y,t)
$
whenever $(x,z),(y,t)\in\mathcal D$ and
$(x,z)\preccurlyeq(y,t)$.



\begin{definition}[\cite{Zedam2023}]
Let $\mathbb P=(X,\unlhd)$ be a psoset. A binary operation $F$ on
$\mathbb P$ is called
\begin{enumerate}
\item[$\mathrm{(i)}$] \emph{commutative} if
$F(x,y)=F(y,x)$ for all $x,y\in X$;
\item[$\mathrm{(ii)}$] \emph{associative} if
$F(x,F(y,z))=F(F(x,y),z)$ for all $x,y,z\in X$;
\item[$\mathrm{(iii)}$] \emph{increasing} if $x\unlhd y$ and
$z\unlhd t$ imply $F(x,z)\unlhd F(y,t)$ for all $x,y,z,t\in X$.
\end{enumerate}
\end{definition}



\begin{definition}[\cite{Zedam2023}]
Let $\mathbb P=(X,\unlhd,0,1)$ be a bounded psoset. A binary
operation $T:X^2\to X$ (respectively, $S:X^2\to X$) is called a
\emph{t-norm} (respectively, a \emph{t-conorm}) on $\mathbb P$ if it
is commutative, associative, and increasing and has neutral element
$1$ (respectively, $0$). Thus,
$
T(x,1)=x
\
\text{(respectively, $S(x,0)=x$)}
$
for every $x\in X$.
\end{definition}

\begin{definition}[\cite{Xiu2025}]
A \emph{nullnorm} with absorbing element $a$ is a commutative,
associative, increasing operation $V:X^2\to X$ satisfying
\begin{equation}\label{eq:boundary}
V(x,0)=x\quad(x\unlhd a),\qquad
V(x,1)=x\quad(a\unlhd x).
\end{equation}
It is \emph{proper} if $a\notin\{0,1\}$.
\end{definition}

The boundary identities and monotonicity imply that $a$ is indeed
absorbing:
$
a=V(a,0)\unlhd V(a,x)\unlhd V(a,1)=a.
$
Hence $V(a,x)=a$ for every $x\in X$.
If $b$ is another absorbing element, then $a=V(a,b)=b$, proving
uniqueness. The endpoint cases $a=0$ and
$a=1$ are precisely t-norms and t-conorms, respectively; we focus
on the proper case. Further results on nullnorms on bounded trellises
can be found in~\cite{Xiu2025}.

\begin{proposition}[\cite{Xiu2025}]\label{propp14}
Let $\mathbb{P}=(X,\unlhd,\wedge,\vee,0,1)$ be a bounded trellis and $V$ be a nullnorm on $\mathbb{P}$ with the absorbing
element $a$.  Then $a\in K$.
\end{proposition}

\begin{definition}[\cite{Xiu2025}]
Let $\mathbb{P}=(X,\unlhd,\wedge,\vee,0,1)$ be a bounded trellis, and
let $V$ be a nullnorm on $\mathbb{P}$ with absorbing element
$a\in K\setminus\{0,1\}$.
\begin{itemize}
\item[$\mathrm{(1)}$]   $V\in\mathcal V_{cwa}$ if $V(x,y)$ is
comparable with $a$ for all $x,y\in X$.
\item[$\mathrm{(2)}$]  $V\in\mathcal V_a$ if
$V(x,y)\notin I_a^1\cup I_a^2$ for all $x,y\in X$.
\end{itemize}
\end{definition}

The following result is implicit in the proof of Theorem~4.1
in~\cite{Xiu2025}.
\begin{proposition}
\label{prop:functions}
Suppose that $f:X\to[a,1]$ and $g:X\to[0,a]$ are order-preserving,
$f(x)=x$ for $x\in[a,1]$, and $g(x)=x$ for $x\in[0,a]$.
Then
\begin{equation}\label{eq:collapse}
f(x)=a\quad(x\in[0,a]\cup I_a^1),\qquad
g(x)=a\quad(x\in[a,1]\cup I_a^2).
\end{equation}
Moreover, for any \( x, y \in X \setminus I_a \),
\begin{equation}\label{eq7}
f(x)=f(y)\ \text{and}\ g(x)=g(y)
\quad\Leftrightarrow \quad x=y.
\end{equation}
\end{proposition}

\section{Construction and Complete Representation}\label{sec:main}
\begin{definition}
Let $\mathbb P=(X,\unlhd,\wedge,\vee,0,1)$ be a bounded trellis
and fix $a\in X\setminus\{0,1\}$. Denote by
$\mathcal V$ the class of all proper nullnorms on
$\mathbb P$ with absorbing element $a$, and define
$
\mathcal V^{123}
=\bigl\{V\in\mathcal V\mid
V(I_a^3\times I_a^3)
\subseteq I_a^1\cup I_a^2\cup I_a^3\bigr\}.
$
\end{definition}

\begin{lemma}\label{lem:regional-sets123}
Let $a\in K\setminus\{0,1\}$ and set
$
A_a=[0,a]\cup I_a^1,
\
B_a=[a,1]\cup I_a^2.
$
Then
$
I_a=I_a^1\mathbin{\cup}I_a^2\mathbin{\cup}I_a^3,
\
A_a\cap B_a=\{a\}.
$
Moreover, $A_a$ is a lower set and $B_a$ is an upper set.
Consequently, $(A_a\times X)\cup(X\times A_a)$ is a lower subset of
$X^2$, whereas $(B_a\times X)\cup(X\times B_a)$ is an upper subset
of $X^2$ with respect to $\preccurlyeq$.
\end{lemma}

\begin{proof}
If an element belonged to both $I_a^1$ and $I_a^2$, it could be used
simultaneously as the elements $y$ and $z$ in the forbidden comparison
in the definition of $K$. Hence $I_a^1\cap I_a^2=\varnothing$, and the
first two assertions follow from the definitions.

Suppose that $v\in A_a$ and $u\unlhd v$. Adjoining the comparison
$u\unlhd v$ to a reachability chain from $v$ to $a$ shows that
$u\lesssim a$. If $u\parallel a$, then $u\in I_a^1$; if
$u\unlhd a$, then $u\in[0,a]$. The only remaining possibility is
$a\lhd u$. In that case $v\ne a$ by antisymmetry, so
$u\in\mathopen{]}a,1\mathclose{[}$ and
$v\in\mathopen{]}0,a\mathclose{[}\cup I_a^1$, contradicting the
definition of $K$. Therefore $u\in A_a$, and $A_a$ is a lower set.
The proof that $B_a$ is an upper set is dual. The assertions about the
two subsets of $X^2$ follow immediately.
\end{proof}

We now state the main theorem, which gives a complete characterization of   all proper nullnorms  on a bounded trellis.

%

First, define
\[
\M = \big(([0,a]\cup I_a^1)\times X\big)\cup\big(X\times([0,a]\cup I_a^1)\big),
\]
\[
\N = \big(([a,1]\cup I_a^2)\times X\big)\cup\big(X\times([a,1]\cup I_a^2)\big),
\qquad \G = I_a^3\times I_a^3,
\]
and set
\[
\mathcal L_a=(X\times I_a^3)\cup(I_a^3\times X).
\]
By Lemma \ref{lem:regional-sets123},    if $a\in K\setminus\{0,1\}$, then  $\M$ is a lower subset of
$X^2$, whereas $\N$ is an upper subset
of $X^2$ with respect to $\preccurlyeq$.


 Given a t-conorm $S:[0,a]^2\to[0,a]$, a t-norm
$T:[a,1]^2\to[a,1]$, order-preserving maps
$f:X\to[a,1]$ and $g:X\to[0,a]$, and a commutative, increasing
function $H:\mathbf G\to X$, define the mixed interaction function
$\Lambda:\mathcal L_a\to X$ as follows:
\begin{equation}\label{eq1}
\Lambda(u,v)=
\begin{cases}
S(g(u),g(v)), & (u,v)\in\M\cap\mathcal L_a,\\
T(f(u),f(v)), & (u,v)\in\N\cap\mathcal L_a,\\
H(u,v), & (u,v)\in\G.
\end{cases}
\end{equation}
Every pair in $\mathcal L_a$ belongs to at least one of the three
displayed regions. Whenever condition~\textup{(C1)} below holds,
formula~\eqref{eq1} defines a well-defined symmetric function
$\Lambda:\mathcal L_a\to X$. Indeed, $\G$ is disjoint from
$\M\cup\N$, whereas the first two expressions are both equal to $a$
on $\M\cap\N\cap\mathcal L_a$ by
Proposition~\ref{prop:functions}. Moreover, $\mathcal L_a$ and all
three regions are symmetric, so commutativity of $S$, $T$, and $H$
makes $\Lambda$ symmetric on its domain.

\begin{theorem}\label{thm:construction}
Let $\mathbb{P}$ be a bounded trellis and $a\in X\setminus\{0,1\}$.
Let $S:[0,a]^2\to[0,a]$ be a t-conorm,
$T:[a,1]^2\to[a,1]$ a t-norm, $f:X\to[a,1]$ and
$g:X\to[0,a]$ order-preserving maps, and $H:\G\to X$ a commutative
and increasing function. Let $\Lambda:\mathcal L_a\to X$ be the mixed
interaction function given by~\eqref{eq1}. Assume that the
following conditions hold whenever the displayed variables belong to
the indicated domains:
\begin{itemize}
\item[(C1)] $f(x)=x$ for $x\in[a,1]$, and $g(x)=x$ for $x\in[0,a]$;
\item[(C2)] For all $(x,y)\in\G$,
$
S(g(x),g(y)) = g(H(x,y)),\qquad T(f(x),f(y)) = f(H(x,y));
$
\item[(C3)] If $(x,z)\in\M$, $(y,t)\in\G$, and
$(x,z)\preccurlyeq(y,t)$, then
$
S(g(x),g(z))\unlhd H(y,t);
$
\item[(C4)] If $(x,z)\in\G$, $(y,t)\in\N$, and
$(x,z)\preccurlyeq(y,t)$, then
$
H(x,z)\unlhd T(f(y),f(t));
$
\item[(C5)] For all $x,y,z\in I_a^3$,
$
\Lambda(H(x,y),z) = \Lambda(x,H(y,z)).
$
\end{itemize}
Then $a\in K\setminus\{0,1\}$
if and only if the binary operation
$V:X\times X\to X$ defined by
\begin{equation}\label{eq2}
V(x,y) =
\begin{cases}
S(g(x),g(y)), & (x,y)\in\M,\\
T(f(x),f(y)), & (x,y)\in\N,\\
H(x,y), & (x,y)\in\G.
\end{cases} 
\end{equation}
is a proper nullnorm on $\mathbb{P}$ with absorbing element $a$.
%

\end{theorem}
\begin{proof}
Suppose first that $a\in K\setminus\{0,1\}$, and define $V$
by~\eqref{eq2}.

\textit{Well-definedness.}
Since
$
X=([0,a]\cup I_a^1)\cup([a,1]\cup I_a^2)\cup I_a^3,
$
the sets $\M$, $\N$, and $\G$ cover $X^2$.
If $(x,y)\in\M\cap\N$, at least one coordinate belongs to
$[0,a]\cup I_a^1$ and at least one coordinate belongs to
$[a,1]\cup I_a^2$. By Proposition~\ref{prop:functions}, one of $f(x)$ and
$f(y)$ equals $a$, and one of $g(x)$ and $g(y)$ equals $a$. Therefore
$S(g(x),g(y))=a=T(f(x),f(y))$, so the two formulas agree.
Comparing~\eqref{eq1} and~\eqref{eq2} now gives
$V|_{\mathcal L_a}=\Lambda$.

\textit{Absorbing element.}
For $x\unlhd a$, we have $x\in[0,a]$ and hence
\[
V(x,0)=S(g(x),g(0))=S(x,0)=x.
\]
For $a\unlhd x$, we have $x\in[a,1]$ and hence
\[
V(x,1)=T(f(x),f(1))=T(x,1)=x.
\]
Thus $a$ is the absorbing element.

\textit{Commutativity.}
The functions $S$, $T$, and $H$ are commutative, and the sets $\M$, $\N$,
and $\G$ are symmetric. Hence $V$ is commutative.

\textit{Increasingness.}
Let $(x,z)\preccurlyeq(y,t)$.  If both pairs belong to the
same region among $\M$, $\N$, and $\G$, the desired inequality follows
from the increasingness of $S$, $T$, or $H$ and the order-preserving
properties of $f$ and $g$.


 Note that $\M$ is a lower subset of
$X^2$ and $\N$ is an upper subset,
a comparison from $\G$ to $\M$ or
from $\N$ to $\G$ is impossible.
 If $(x,z)\in\N$ and
$(y,t)\in\M$, then the lower-set property of $\M$ and the upper-set
property of $\N$ imply
$
(x,z),(y,t)\in\M\cap\N.
$
Both values are therefore equal to $a$.  It remains to consider
the following cross-region cases for $(x,z)$ and $(y,t)$.
\begin{enumerate}
\item If $(x,z)\in\M$ and $(y,t)\in\N$, then
$
V(x,z)\unlhd a\unlhd V(y,t),
$
and $a\in X^{mtr}$ gives $V(x,z)\unlhd V(y,t)$.
\item The transition from $\M$ to $\G$ is precisely
condition~\textup{(C3)}.
\item The transition from $\G$ to $\N$ is precisely
condition~\textup{(C4)}.
\end{enumerate}

 Thus $V$ is increasing.

\textit{Associativity.}
We prove that $V(V(x,y),z)=V(x,V(y,z))$ for all $x,y,z\in X$.
First,
we prove that for all $(u,v)\in X^2$,
\begin{equation}\label{eq3}
g(V(u,v))=S(g(u),g(v)),\qquad f(V(u,v))=T(f(u),f(v)). 
\end{equation}
If $(u,v)\in\M$, then $V(u,v)=S(g(u),g(v))\in[0,a]$. By (C1),
\[
g(V(u,v))=g(S(g(u),g(v)))=S(g(u),g(v)).
\]
Moreover, at least one of $u$ and $v$ belongs to $[0,a]\cup I_a^1$,
so Proposition~\ref{prop:functions} yields $f(u)=a$ or $f(v)=a$. Hence
$T(f(u),f(v))=a=f(V(u,v))$.
The argument for \(\N\) is analogous.
If $(u,v)\in\G$, then (C2) gives the identities.

We now compare the two iterated values.

If $x,y,z\in I_a^3$, then $V(x,y)=H(x,y)$ and
$V(y,z)=H(y,z)$. Moreover,
\[
(H(x,y),z)\in X\times I_a^3,
\qquad (x,H(y,z))\in I_a^3\times X,
\]
so both pairs belong to $\mathcal L_a$. Therefore
$V|_{\mathcal L_a}=\Lambda$ and condition \textup{(C5)} give
\[
V(V(x,y),z)=\Lambda(H(x,y),z)=\Lambda(x,H(y,z))=V(x,V(y,z)).
\]

Suppose that at least one of $x,y,z$ does not belong to $I_a^3$.
If $x\notin I_a^3$, then $(x,V(y,z))\notin\G$, while
$(x,y)\notin\G$ implies $V(x,y)\in[0,a]\cup[a,1]$ and hence
$(V(x,y),z)\notin\G$. The case $z\notin I_a^3$ is symmetric. If
$x,z\in I_a^3$ and $y\notin I_a^3$, then neither $(x,y)$ nor $(y,z)$
belongs to $\G$, so both inner values lie in
$[0,a]\cup[a,1]$ and again neither outer pair belongs to $\G$.
In all cases, both iterated values are produced by an interval branch;
therefore they lie in $[0,a]\cup[a,1]\subseteq X\setminus I_a$.

Using~\eqref{eq3} and the associativity of $S$, we obtain
\[
g(V(V(x,y),z))=S(S(g(x),g(y)),g(z))=S(g(x),S(g(y),g(z)))=g(V(x,V(y,z))).
\]
The analogous equality holds for $f$ by associativity of $T$.
By~\eqref{eq7} in Proposition~\ref{prop:functions}, the desired equality follows.

Thus $V$ is associative.

Hence $V$ is a nullnorm on $\mathbb{P}$ with absorbing element $a$.

Conversely, if the operation in~\eqref{eq2} is a nullnorm with absorbing
element $a$, Proposition~\ref{propp14} yields $a\in K$. Since $V$ is
proper by assumption, $a\in K\setminus\{0,1\}$.

\end{proof}

\begin{theorem}
\label{thm:representation}
Let $\mathbb{P}=(X,\unlhd,\wedge,\vee,0,1)$ be a bounded trellis,
let $a\in X\setminus\{0,1\}$, and let $V:X^2\to X$ be a binary
operation.
Then the following statements are equivalent:
\begin{enumerate}
\item[$\mathrm{(1)}$] $V$ is a proper nullnorm on $\mathbb P$ with absorbing
 element $a$.
\item[$\mathrm{(2)}$] $a\in K\setminus\{0,1\}$ and there exist   $S,T,f,g$ and $H$
satisfying all the hypotheses and conditions~\textup{(C1)--(C5)} of
Theorem~\ref{thm:construction} such that $V$ is given by
formula~\eqref{eq2}.
\end{enumerate}
Moreover, when these statements hold, the representing components are uniquely
recovered from $V$ by
\[
S=V|_{[0,a]^2},\qquad T=V|_{[a,1]^2},\qquad
f(x)=V(x,1),\qquad g(x)=V(x,0),\qquad H=V|_{\G}.
\]
\end{theorem}
\begin{proof}
\noindent\textbf{$(1)\Rightarrow(2)$.}
Assume that $V$ is a nullnorm with absorbing element $a$.
By Proposition~\ref{propp14}, $a\in K$. Define
\[
S=V|_{[0,a]^2},\qquad T=V|_{[a,1]^2},\qquad
f(x)=V(x,1),\qquad g(x)=V(x,0),
\]
and set $H=V|_{\G}$.

The restrictions $S$ and $T$ are, respectively, a t-conorm on
$[0,a]$ and a t-norm on $[a,1]$. Indeed, the increasingness and the absorbing
property show that these intervals are closed under the corresponding
restrictions, and all the remaining axioms are inherited from $V$.
Moreover, $f$ and $g$ are order-preserving and
\[
g(x)=V(x,0)\unlhd V(x,a)=a,
\qquad
a=V(x,a)\unlhd V(x,1)=f(x).
\]
Thus $g:X\to[0,a]$ and $f:X\to[a,1]$. The boundary conditions give
\[
f(x)=x\quad(x\in[a,1]),
\qquad
g(x)=x\quad(x\in[0,a]),
\]
so condition~\textup{(C1)} holds and
Proposition~\ref{prop:functions} applies. The function $H:\G\to X$
is commutative and increasing because it is a restriction of $V$.

For all $u,v\in X$, associativity and commutativity, together with
$V(0,0)=0$ and $V(1,1)=1$, yield
\begin{equation}\label{eq:boundary-homomorphisms}
\begin{aligned}
g(V(u,v))
 &=V(V(u,v),0)
  =V(V(u,0),V(v,0))
  =S(g(u),g(v)),\\
f(V(u,v))
 &=V(V(u,v),1)
  =V(V(u,1),V(v,1))
  =T(f(u),f(v)).
\end{aligned}
\end{equation}

We next recover $V$ on the three regions. If $(u,v)\in\M$, one
coordinate belongs to $[0,a]\cup I_a^1$. If, for example,
$u\in[0,a]$, then $V(u,v)\unlhd V(a,v)=a$. If $u\in I_a^1$,
Proposition~\ref{prop:functions} gives $f(u)=V(u,1)=a$, and hence
$V(u,v)\unlhd V(u,1)=a$. The case in which the second coordinate
belongs to $[0,a]\cup I_a^1$ follows by commutativity. Therefore
$V(\M)\subseteq[0,a]$, and~\eqref{eq:boundary-homomorphisms} gives
\[
V(u,v)=g(V(u,v))=S(g(u),g(v))
\qquad ((u,v)\in\M).
\]
Dually,
\[
V(\N)\subseteq[a,1],\qquad
V(u,v)=f(V(u,v))=T(f(u),f(v))
\qquad ((u,v)\in\N).
\]
On $\G$, we have $V=H$ by definition. Hence formula~\eqref{eq2}
recovers $V$ on all of $X^2$, and the mixed interaction function
satisfies
\[
\Lambda=V|_{\mathcal L_a}.
\]
For $(x,y)\in\G$, the two identities
in~\eqref{eq:boundary-homomorphisms} are precisely
condition~\textup{(C2)}.

Conditions (C3) and (C4) follow from the increasingness of $V$.



For (C5), let $x,y,z\in I_a^3$. Then $H(x,y)=V(x,y)$,
$H(y,z)=V(y,z)$, and, by the definition of $\Lambda$,
\[
\Lambda(H(x,y),z)=V(V(x,y),z),\qquad \Lambda(x,H(y,z))=V(x,V(y,z)).
\]
Associativity of $V$ then yields the desired equality.

\noindent\textbf{$(2)\Rightarrow(1)$.}
This implication follows directly from Theorem~\ref{thm:construction}.

Finally, the displayed recovery formulas in the theorem follow from
the definitions in the forward implication. Conversely, if $V$ is
given by~\eqref{eq2}, condition~\textup{(C1)} and
Proposition~\ref{prop:functions} show directly that its two interval
restrictions and its functions $x\mapsto V(x,1)$ and
$x\mapsto V(x,0)$ are exactly $S,T,f$  and $g$, while its restriction to
$\G$ is $H$. Thus the representing components are unique.
\end{proof}

\begin{remark}
If the absorbing element is $a=0$, then a nullnorm $V$ is precisely a
t-norm, whereas if $a=1$, then $V$ is precisely a t-conorm. Thus, the
endpoint cases are covered by t-norms and t-conorms, respectively.
Together with Theorem~\ref{thm:representation}, which characterizes all
proper nullnorms, this observation yields a complete representation of
all nullnorms on bounded trellises. Accordingly, in this manuscript,  we use the title
``A complete representation theorem for all nullnorms on bounded
trellises.''
\end{remark}



\begin{remark}\label{rem:scope}
The representation of the class $\mathcal V_a$ obtained
in~\cite{Xiu2025} imposes a restriction on the range of $H$.
No such range restriction is imposed in the present theorem. More
precisely, the function
$
H:I_a^3\times I_a^3\rightarrow X
$
is allowed to take values in any part of $X$, including $I_a^1$ and
$I_a^2$, provided that conditions \textup{(C2)--(C5)} are satisfied.
This unrestricted range explains why condition \textup{(C5)} must be
formulated in terms of the mixed interaction function $\Lambda$. For example,
if $H(x,y)\in I_a^1\cup I_a^2$, then $H(x,y)\notin I_a^3$, and hence
the expression $H(H(x,y),z)$
is not defined, because the domain of $H$ is only $I_a^3\times I_a^3$.
In contrast, each pair occurring in \textup{(C5)} belongs to
$\mathcal L_a$, since
\[
(H(x,y),z)\in X\times I_a^3,
\qquad
(x,H(y,z))\in I_a^3\times X.
\]
Thus both evaluations of $\Lambda$ are defined and automatically select
the appropriate branches of the piecewise construction. Consequently,
condition \textup{(C5)} requires
\[
\Lambda(H(x,y),z)=\Lambda(x,H(y,z)),
\qquad x,y,z\in I_a^3,
\]
which is precisely the associativity condition needed for triples
entirely contained in $I_a^3$.

\end{remark}



The following example illustrates the construction method described in
Theorem~\ref{thm:construction}. In this example, the range of $H$
intersects each of the sets $[0,a]$, $[a,1]$, $I_a^1$, $I_a^2$, and
$I_a^3$.
\begin{example}\label{ex31}
Let   $X=\{0,l,m,a,n,v,b,d,p,w,s,t,e,1\}.$
The pseudo-order shown in Figure~\ref{figure3} is specified as follows.
The subset $\{0,l,m,a,n,v,1\}$ is a chain in the displayed order,
$0\unlhd x\unlhd1$ for every $x\in X$, and the additional
nontrivial comparisons are
\[
l\unlhd b,\quad b\unlhd m,\quad b\unlhd n,\qquad
m\unlhd d,\quad n\unlhd d,\quad d\unlhd v,
\]
and
\[
l\unlhd e,\quad m\unlhd e,\quad e\unlhd n,\quad e\unlhd v.
\]
There are no other nontrivial comparisons; in particular,
$b\parallel a$ and $d\parallel a$. Every pair has a greatest lower
bound and a least upper bound, so
$\mathbb{P}=(X,\unlhd,\wedge,\vee,0,1)$ is a bounded trellis.
Moreover, $a$ is middle-transitive and no element of
$\mathopen{]}a,1\mathclose{[}\cup I_a^2$ lies below an element of
$\mathopen{]}0,a\mathclose{[}\cup I_a^1$; hence $a\in K$. The five
regions determined by $a$ are
\[
\begin{aligned}
{[0,a]}&=\{0,l,m,a\},\qquad [a,1]=\{a,n,v,1\},\\
I_a^1&=\{b\},\qquad I_a^2=\{d\},\qquad
I_a^3=\{p,w,s,t,e\}.
\end{aligned}
\]

Define the order-preserving maps $f:X\to[a,1]$ and
$g:X\to[0,a]$ by the values listed in
Table~\ref{tab:ex31-boundary-maps}.
\begin{table}[H]
\centering
\caption{The order-preserving maps $f$ and $g$ in
Example~\ref{ex31}.
}
\label{tab:ex31-boundary-maps}
\small
\setlength{\tabcolsep}{3.5pt}
\begin{tabular}{c|*{14}{c}}
\hline
$x$&$0$&$l$&$m$&$a$&$n$&$v$&$b$&$d$&$p$&$w$&$s$&$t$&$e$&$1$\\
\hline
$f(x)$&$a$&$a$&$a$&$a$&$n$&$v$&$a$&$n$&$n$&$n$&$v$&$v$&$a$&$1$\\
$g(x)$&$0$&$l$&$m$&$a$&$a$&$a$&$m$&$a$&$l$&$l$&$m$&$m$&$a$&$a$\\
\hline
\end{tabular}
\end{table}
Identify the two boundary
chains with $\{0,1,2,3\}$ by
\[
0\leftrightarrow0, l\leftrightarrow1, m\leftrightarrow2,
\ a\leftrightarrow3
\quad\text{and}\quad
a\leftrightarrow0, n\leftrightarrow1, v\leftrightarrow2,
\ 1\leftrightarrow3.
\]
For the corresponding indices, define
\[
S(i,j)=\min\{3,i+j\},\qquad
T(i,j)=\max\{0,i+j-3\}.
\]
The complete operation tables of $S$ and $T$ are given in
Table~\ref{tab:ex31-st}.
\begin{table}[H]
\centering
\caption{The \L ukasiewicz t-conorm
$S$ on $[0,a]$ and t-norm $T$ on $[a,1]$ in
Example~\ref{ex31}.
}
\label{tab:ex31-st}
\small
\renewcommand{\arraystretch}{1.05}
\begin{tabular}{c|cccc}
$S$&$0$&$l$&$m$&$a$\\ \hline
$0$&$0$&$l$&$m$&$a$\\
$l$&$l$&$m$&$a$&$a$\\
$m$&$m$&$a$&$a$&$a$\\
$a$&$a$&$a$&$a$&$a$
\end{tabular}
\qquad
\begin{tabular}{c|cccc}
$T$&$a$&$n$&$v$&$1$\\ \hline
$a$&$a$&$a$&$a$&$a$\\
$n$&$a$&$a$&$a$&$n$\\
$v$&$a$&$a$&$n$&$v$\\
$1$&$a$&$n$&$v$&$1$
\end{tabular}
\end{table}
Thus $S$ is the four-element \L ukasiewicz t-conorm on $[0,a]$,
and $T$ is the four-element \L ukasiewicz t-norm on $[a,1]$.
Define $H:I_a^3\times I_a^3\to X$ by the values listed in
Table~\ref{tab:ex31-h}.
\begin{table}[H]
\centering
\caption{The function $H$ in
Example~\ref{ex31}.}
\label{tab:ex31-h}
\small
\renewcommand{\arraystretch}{1.05}
\begin{tabular}{c|ccccc}
\hline
$H$&$p$&$w$&$s$&$t$&$e$\\ \hline
$p$&$m$&$m$&$e$&$a$&$a$\\
$w$&$m$&$b$&$a$&$a$&$a$\\
$s$&$e$&$a$&$n$&$n$&$a$\\
$t$&$a$&$a$&$n$&$d$&$a$\\
$e$&$a$&$a$&$a$&$a$&$a$\\
\hline
\end{tabular}
\end{table}
Its range is
\[
H(I_a^3\times I_a^3)=\{m,n,a,b,d,e\}.
\]
Consequently, its range intersects each of the five regions determined
by $a$: $m\in[0,a]$,
$n\in[a,1]$, $b\in I_a^1$, $d\in I_a^2$, and $e\in I_a^3$.
\begin{samepage}
The mixed interaction function $\Lambda$ on
$\mathcal L_a=(X\times I_a^3)\cup(I_a^3\times X)$ is displayed in
Table~\ref{tab:ex31-lambda}. Its rows are indexed by $x\in X$ and its
columns by $y\in I_a^3=\{p,w,s,t,e\}$. Since $\Lambda$ is symmetric,
this table determines its values on all of $\mathcal L_a$.
\begin{table}[H]
\centering
\caption{The mixed interaction function $\Lambda$ on
$X\times I_a^3$ in Example~\ref{ex31}.
}
\label{tab:ex31-lambda}
\small
\renewcommand{\arraystretch}{1.05}
\setlength{\tabcolsep}{7pt}
\begin{tabular}{c|ccccc}
\hline
$\Lambda$&$p$&$w$&$s$&$t$&$e$\\ \hline
$0$&$l$&$l$&$m$&$m$&$a$\\
$l$&$m$&$m$&$a$&$a$&$a$\\
$m$&$a$&$a$&$a$&$a$&$a$\\
$a$&$a$&$a$&$a$&$a$&$a$\\
$n$&$a$&$a$&$a$&$a$&$a$\\
$v$&$a$&$a$&$n$&$n$&$a$\\
$b$&$a$&$a$&$a$&$a$&$a$\\
$d$&$a$&$a$&$a$&$a$&$a$\\
$p$&$m$&$m$&$e$&$a$&$a$\\
$w$&$m$&$b$&$a$&$a$&$a$\\
$s$&$e$&$a$&$n$&$n$&$a$\\
$t$&$a$&$a$&$n$&$d$&$a$\\
$e$&$a$&$a$&$a$&$a$&$a$\\
$1$&$n$&$n$&$v$&$v$&$a$\\
\hline
\end{tabular}
\end{table}
\end{samepage}

We now verify that $f$, $g$, $S$, $T$, $H$, and $\Lambda$
satisfy the hypotheses of
Theorem~\ref{thm:construction}. Condition~\textup{(C1)} follows from
Table~\ref{tab:ex31-boundary-maps}.

To verify \textup{(C2)}, put $P=\{p,w\}$ and $R=\{s,t\}$.
Table~\ref{tab:ex31-boundary-maps} gives
\[
(g(x),f(x))=
\begin{cases}
(l,n),&x\in P,\\
(m,v),&x\in R,\\
(a,a),&x=e.
\end{cases}
\]
Using the operation tables of $S$ and $T$, we obtain the transported
input values in Table~\ref{tab:ex31-transported-pairs}. The middle
column also covers $R\times P$, since $S$ and $T$ are commutative.
\begin{table}[H]
\centering
\caption{Pairs of transported values used to verify
condition~\textup{(C2)} in Example~\ref{ex31}.}
\label{tab:ex31-transported-pairs}
\small
\begin{tabular}{c|ccc}
\hline
\text{input class}&$P^2$&$(P\times R)\cup(R\times P)$&$R^2$\\ \hline
$(S(g(\cdot),g(\cdot)),T(f(\cdot),f(\cdot)))$
&$(m,a)$&$(a,a)$&$(a,n)$\\ \hline
\end{tabular}
\end{table}

On the other hand, Table~\ref{tab:ex31-h} gives
\[
H(P^2)\subseteq\{m,b\},\qquad
H\bigl((P\times R)\cup(R\times P)\bigr)\subseteq\{a,e\},
\qquad H(R^2)\subseteq\{n,d\}.
\]
For these possible output values, Table~\ref{tab:ex31-boundary-maps}
gives
\[
\begin{aligned}
(g(h),f(h))&=(m,a) &&\text{for }h\in\{m,b\},\\
(g(h),f(h))&=(a,a) &&\text{for }h\in\{a,e\},\\
(g(h),f(h))&=(a,n) &&\text{for }h\in\{n,d\}.
\end{aligned}
\]
Thus, for every $(x,y)\in(P\cup R)^2$,
\[
\bigl(S(g(x),g(y)),T(f(x),f(y))\bigr)
=\bigl(g(H(x,y)),f(H(x,y))\bigr).
\]
If $x=e$ or $y=e$, then $H(x,y)=a$ and both sides of this equality
are $(a,a)$. These cases exhaust
$I_a^3\times I_a^3$, and therefore condition~\textup{(C2)} holds.


Since the elements of $I_a^3$ are pairwise incomparable, the
componentwise order on $I_a^3\times I_a^3$ contains no nontrivial
comparisons. Hence the function $H$ defined above is automatically
increasing. Retaining the notation $P=\{p,w\}$ and $R=\{s,t\}$, define
\[
E=(\{e\}\times I_a^3)\cup(I_a^3\times\{e\}).
\]
\begin{samepage}
For each
$\mathcal C\in
\bigl\{P^2,(P\times R)\cup(R\times P),R^2,E\bigr\}$, define
\[
L(\mathcal C)=
\left\{
S(g(u),g(v))\,\middle|\,
\begin{array}{l}
(u,v)\in\M,\ \text{and}\\
(u,v)\preccurlyeq(x,y)
\text{ for some }(x,y)\in\mathcal C
\end{array}
\right\},
\]
and
\[
U(\mathcal C)=
\left\{
T(f(u),f(v))\,\middle|\,
\begin{array}{l}
(u,v)\in\N,\ \text{and}\\
(x,y)\preccurlyeq(u,v)
\text{ for some }(x,y)\in\mathcal C
\end{array}
\right\}.
\]
\end{samepage}
A direct inspection of the comparisons in Figure~\ref{figure3} yields
Table~\ref{tab:ex31-order-bounds}.
\begin{table}[H]
\centering
\caption{Values used to verify conditions~\textup{(C3)} and
\textup{(C4)} in Example~\ref{ex31}.}
\label{tab:ex31-order-bounds}
\small
\begin{tabular}{c|c|c|c}
\hline
$\mathcal C$&$H(\mathcal C)$&$L(\mathcal C)$&$U(\mathcal C)$\\ \hline
$P^2$&$\{m,b\}$&$\{0,l\}$&$\{n,1\}$\\
$(P\times R)\cup(R\times P)$
    &$\{a,e\}$&$\{0,l,m\}$&$\{n,v,1\}$\\
$R^2$&$\{n,d\}$&$\{0,m\}$&$\{v,1\}$\\
$E$&$\{a\}$&$\{0,l,m,a\}$&$\{a,n,v,1\}$\\ \hline
\end{tabular}
\end{table}
For every set $\mathcal C$ appearing in
Table~\ref{tab:ex31-order-bounds},
\[
r\unlhd h\unlhd u
\qquad
\bigl(r\in L(\mathcal C),\
      h\in H(\mathcal C),\
      u\in U(\mathcal C)\bigr).
\]
Therefore, conditions \textup{(C3)} and \textup{(C4)} are satisfied.


Finally, let $x,y,z\in I_a^3$. If
$H(x,y)\in\{m,b\}$, then
$\Lambda(H(x,y),z)=S(m,g(z))=a$. If
$H(x,y)\in\{n,d\}$, then
$\Lambda(H(x,y),z)=T(n,f(z))=a$. If
$H(x,y)\in\{a,e\}$, then
$\Lambda(H(x,y),z)=a$. Hence
\[
\Lambda(H(x,y),z)=a.
\]
Applying the same argument to $H(y,z)$ and using the symmetry of
$\Lambda$, we obtain
\[
\Lambda(x,H(y,z))
 =\Lambda(H(y,z),x)=a.
\]
Therefore condition~\textup{(C5)} holds.
 All assumptions of
Theorem~\ref{thm:construction} are satisfied, so equation~\eqref{eq2}
defines a nullnorm on $\mathbb{P}$ with absorbing element $a$.

The resulting nullnorm $V$ is displayed in
Table~\ref{tab:ex31-nullnorm}.
\begin{table}[H]
\centering
\caption{The nullnorm $V$ constructed via Theorem \ref{thm:construction}.
}
\label{tab:ex31-nullnorm}
\footnotesize
\renewcommand{\arraystretch}{0.95}
\setlength{\tabcolsep}{2.8pt}
\begin{tabular}{c|*{14}{c}}
\hline
$V$&$0$&$l$&$m$&$a$&$n$&$v$&$b$&$d$&$p$&$w$&$s$&$t$&$e$&$1$\\ \hline
$0$&$0$&$l$&$m$&$a$&$a$&$a$&$m$&$a$&$l$&$l$&$m$&$m$&$a$&$a$\\
$l$&$l$&$m$&$a$&$a$&$a$&$a$&$a$&$a$&$m$&$m$&$a$&$a$&$a$&$a$\\
$m$&$m$&$a$&$a$&$a$&$a$&$a$&$a$&$a$&$a$&$a$&$a$&$a$&$a$&$a$\\
$a$&$a$&$a$&$a$&$a$&$a$&$a$&$a$&$a$&$a$&$a$&$a$&$a$&$a$&$a$\\
$n$&$a$&$a$&$a$&$a$&$a$&$a$&$a$&$a$&$a$&$a$&$a$&$a$&$a$&$n$\\
$v$&$a$&$a$&$a$&$a$&$a$&$n$&$a$&$a$&$a$&$a$&$n$&$n$&$a$&$v$\\
$b$&$m$&$a$&$a$&$a$&$a$&$a$&$a$&$a$&$a$&$a$&$a$&$a$&$a$&$a$\\
$d$&$a$&$a$&$a$&$a$&$a$&$a$&$a$&$a$&$a$&$a$&$a$&$a$&$a$&$n$\\
$p$&$l$&$m$&$a$&$a$&$a$&$a$&$a$&$a$&$m$&$m$&$e$&$a$&$a$&$n$\\
$w$&$l$&$m$&$a$&$a$&$a$&$a$&$a$&$a$&$m$&$b$&$a$&$a$&$a$&$n$\\
$s$&$m$&$a$&$a$&$a$&$a$&$n$&$a$&$a$&$e$&$a$&$n$&$n$&$a$&$v$\\
$t$&$m$&$a$&$a$&$a$&$a$&$n$&$a$&$a$&$a$&$a$&$n$&$d$&$a$&$v$\\
$e$&$a$&$a$&$a$&$a$&$a$&$a$&$a$&$a$&$a$&$a$&$a$&$a$&$a$&$a$\\
$1$&$a$&$a$&$a$&$a$&$n$&$v$&$a$&$n$&$n$&$n$&$v$&$v$&$a$&$1$\\
\hline
\end{tabular}
\end{table}
\end{example}

\begin{figure}[htbp]
  \centering
  \begin{tikzpicture}[scale=0.92]
    \coordinate (zero) at (0,0);
    \coordinate (l)    at (0,1);
    \coordinate (m)    at (0,2);
    \coordinate (a)    at (0,3);
    \coordinate (n)    at (0,4);
    \coordinate (v)    at (0,5);
    \coordinate (one)  at (0,6);
    \coordinate (b)    at (-1.7,1.55);
    \coordinate (e)    at (-1.9,3.0);
    \coordinate (d)    at (1.7,4.45);
    \coordinate (p)    at (-4.8,1.8);
    \coordinate (w)    at (-3.8,3.9);
    \coordinate (s)    at (3.8,2.0);
    \coordinate (t)    at (4.8,4.0);

    \draw (zero)--(l)--(m)--(a)--(n)--(v)--(one);
    \draw (l)--(b)--(m);
    \draw (b) to[bend left=18] (n);
    \draw (m) to[bend right=16] (d);
    \draw (n)--(d)--(v);
    \draw (m)--(e)--(n);
    \draw (zero)--(p)--(one);
    \draw (zero)--(w)--(one);
    \draw (zero)--(s)--(one);
    \draw (zero)--(t)--(one);

    \draw[dashed] (b) to[bend left=22] (a);
    \draw[dashed] (a) to[bend right=22] (d);

    \foreach \q in {zero,l,m,a,n,v,one,b,d,e,p,w,s,t}
      \fill (\q) circle (2pt);

    \node[below]       at (zero) {$0$};
    \node[right]       at (l) {$l$};
    \node[right]       at (m) {$m$};
    \node[right]       at (a) {$a$};
    \node[right]       at (n) {$n$};
    \node[right]       at (v) {$v$};
    \node[above]       at (one) {$1$};
    \node[left]        at (b) {$b$};
    \node[left]        at (e) {$e$};
    \node[right]       at (d) {$d$};
    \node[left]        at (p) {$p$};
    \node[left]        at (w) {$w$};
    \node[right]       at (s) {$s$};
    \node[right]       at (t) {$t$};
  \end{tikzpicture}
  \caption{Hasse-type diagram of the bounded trellis $\mathbb P$ in
  Example~\ref{ex31}. The dashed curves indicate the missing
  comparisons $b \ntrianglelefteq  a$ and $a \ntrianglelefteq d$.
  }
  \label{figure3}
\end{figure}
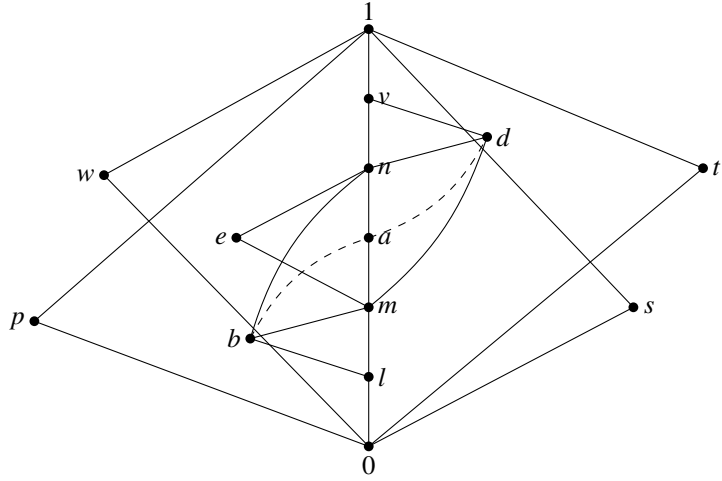


The following example demonstrates the independence of the associativity
condition~\textup{(C5)} and, in particular, shows that this condition
cannot be omitted.
In other words, the component hypotheses and
\textup{(C1)--(C4)} in
Theorem~\ref{thm:construction} do not imply \textup{(C5)}, even
when the underlying trellis is a bounded lattice.

\begin{example}\label{ex:associativity}
Let $X=\{0,a,c,d,1\}$ be the lattice whose three pairwise
incomparable atoms are $a, c$ and $d$. Thus $I_a^3=\{c,d\}$ and $a\in K$.
Take $S=\max$ on $\{0,a\}$ and $T=\min$ on $\{a,1\}$. Define the
remaining functions by Table~\ref{tab:associativity-data}.
\begin{table}[H]
\centering
\caption{The functions $f$, $g$, and $H$ in
Example~\ref{ex:associativity}.}
\label{tab:associativity-data}
\begin{tabular}{c|ccccc}
$x$&$0$&$a$&$c$&$d$&$1$\\\hline
$f(x)$&$a$&$a$&$1$&$1$&$1$\\
$g(x)$&$0$&$a$&$0$&$0$&$a$
\end{tabular}
\qquad
\begin{tabular}{c|cc}
$H$&$c$&$d$\\\hline
$c$&$d$&$c$\\
$d$&$c$&$c$
\end{tabular}
\end{table}
Then $f$ and $g$ are order-preserving and satisfy \textup{(C1)}. Moreover,
$H$ is commutative. Also, it is increasing because the componentwise
relation on $\{c,d\}\times \{c,d\}$ has no comparisons other than reflexive ones.
The first identity in \textup{(C2)} has both sides equal to $0$, and the
second has both sides equal to $1$.

For \textup{(C3)}, suppose that $(x,z)\in\M$ lies below an element of
$\{c,d\}\times \{c,d\}$.
The coordinate of $(x,z)$ that belongs to $[0,a]$ must be $0$, and the
other coordinate is also mapped to $0$ by $g$; hence the lower value is
$S(0,0)=0$.
Dually, every comparison in
\textup{(C4)} has upper value $T(1,1)=1$. Thus \textup{(C3)} and
\textup{(C4)} hold. However,
\[
\Lambda(H(c,c),d)=\Lambda(d,d)=c
\ne d=\Lambda(c,c)=\Lambda(c,H(c,d)).
\]
Thus \textup{(C5)} and associativity fail at $(c,c,d)$. The
well-definedness, absorbing element, commutativity, and monotonicity in
the proof of Theorem~\ref{thm:construction} use only
\textup{(C1)--(C4)}, so they
remain valid for this nonassociative operation.
\end{example}

%

\section{Special Cases and Range-Restricted Subclasses}
\label{sec:specializations}
\subsection{Empty and Singleton Cases of
\texorpdfstring{$I_a^3$}{Ia3}}




Representation theorems for nullnorms on bounded trellises depend
crucially on the values of the nullnorm on the region
$I_a^3\times I_a^3$. For instance, the representation
in~\cite{Xiu2025} uses a function 
$
H:I_a^3\times I_a^3\rightarrow
[0,a]\cup[a,1]\cup I_a^3,
$ 
whereas the representation established in this paper permits 
$
H:I_a^3\times I_a^3\rightarrow X.
$ 
This distinction highlights the important role of $I_a^3$ itself in
the representation of nullnorms on bounded trellises. To further
clarify this role, we specialize Theorem~\ref{thm:construction} to the
cases $I_a^3=\varnothing$ and $I_a^3=\{c\}$.
These cases are
structurally important because the size of $I_a^3$ determines both
the freedom in choosing $H$ and the complexity of condition
\textup{(C5)}. If $I_a^3=\varnothing$, there are no values of $H$ to
specify, and the construction is completely determined by $S,T,f$,
and $g$. If $I_a^3=\{c\}$, the function $H$ is determined by the single
value $H(c,c)$, and the sole instance of condition~\textup{(C5)} follows
from the symmetry of $\Lambda$. Consequently,
a genuinely nontrivial associativity condition can arise only
when $I_a^3$ contains at least two elements. The following two
theorems give the corresponding specializations of
Theorem~\ref{thm:construction}.


\begin{theorem}
\label{thm:empty-core}
Let $\mathbb{P}=(X,\unlhd,\wedge,\vee,0,1)$ be a bounded trellis,
let $a\in X\setminus\{0,1\}$, and suppose that $I_a^3=\varnothing$.
Let $S:[0,a]^2\to[0,a]$ be a t-conorm,
$T:[a,1]^2\to[a,1]$ be a t-norm, and let
$f:X\to[a,1]$ and $g:X\to[0,a]$ be order-preserving maps satisfying
\[
f(x)=x\quad(x\in[a,1]),\qquad
g(x)=x\quad(x\in[0,a]).
\]
Then $a\in K$
if and only if the binary operation
$V:X^2\to X$ defined by
\begin{equation}\label{eq:empty-core}
V(x,y)=
\begin{cases}
S(g(x),g(y)), & (x,y)\in\M,\\
T(f(x),f(y)), & (x,y)\in\N
\end{cases}
\end{equation}
is a nullnorm on $\mathbb{P}$ with absorbing element $a$.
\end{theorem}

\begin{proof}
Since $I_a^3=\varnothing$, we have $\G=\varnothing$. Hence there is
a unique function $H:\G\to X$, and it is automatically commutative and
increasing. Conditions \textup{(C2)--(C5)} of
Theorem~\ref{thm:construction} are vacuous, while \textup{(C1)} is exactly the
displayed condition on $f$ and $g$. Equation~\eqref{eq:empty-core} is
therefore the specialization of~\eqref{eq2}, so the assertion follows
from Theorem~\ref{thm:construction}.
\end{proof}

\begin{remark}
For the fixed absorbing element $a$, if $I_a^3=\varnothing$, then
$
\mathcal V=\mathcal V_a=\mathcal V_{cwa}=\mathcal V^{123}.
$
Indeed, Theorem~\ref{thm:representation} shows that every member of
$\mathcal V$ has the form~\eqref{eq:empty-core}, and hence all of its
values lie in $[0,a]\cup[a,1]$. The defining condition of
$\mathcal V^{123}$ is vacuous, which proves the asserted equality.
\end{remark}

\begin{theorem}
\label{thm:singleton-core}
Let $\mathbb{P}=(X,\unlhd,\wedge,\vee,0,1)$ be a bounded trellis,
let $a\in X\setminus\{0,1\}$, and suppose that $I_a^3=\{c\}$.
Let $S:[0,a]^2\to[0,a]$ be a t-conorm,
$T:[a,1]^2\to[a,1]$ be a t-norm, and let
$f:X\to[a,1]$ and $g:X\to[0,a]$ be order-preserving maps satisfying
$
f(x)=x\quad(x\in[a,1]),\
g(x)=x\quad(x\in[0,a]).
$
Assume that there exists $h\in X$ such that
\begin{align}
g(h)&=S(g(c),g(c)), &
f(h)&=T(f(c),f(c)),                                      \label{eq:singleton-fibres}\\
S(g(x),g(z))&\unlhd h
&&\text{whenever }(x,z)\in\M\text{ and }
  (x,z)\preccurlyeq(c,c),                               \label{eq:singleton-lower}\\
h&\unlhd T(f(y),f(t))
&&\text{whenever }(y,t)\in\N\text{ and }
  (c,c)\preccurlyeq(y,t).                               \label{eq:singleton-upper}
\end{align}
Then $a\in K$ if and only if the binary operation
$V:X^2\to X$ defined by
\begin{equation}\label{eq:singleton-core}
V(x,y)=
\begin{cases}
S(g(x),g(y)), & (x,y)\in\M,\\
T(f(x),f(y)), & (x,y)\in\N,\\
h,           & (x,y)=(c,c)
\end{cases}
\end{equation}
is a nullnorm on $\mathbb{P}$ with absorbing element $a$.
\end{theorem}

\begin{proof}
Define $H:\G\to X$ by $H(c,c)=h$. Since
$\G=\{(c,c)\}$, the function $H$ is automatically commutative and
increasing. Equation~\eqref{eq:singleton-fibres} is precisely
condition \textup{(C2)} of Theorem~\ref{thm:construction}, whereas
\eqref{eq:singleton-lower} and~\eqref{eq:singleton-upper} are its
conditions \textup{(C3)} and \textup{(C4)}, respectively. The only
instance of \textup{(C5)} is
\[
\Lambda(H(c,c),c)=\Lambda(h,c)
=\Lambda(c,h)=\Lambda(c,H(c,c)),
\]
which follows from the symmetry of
$\Lambda$. Thus all the hypotheses of Theorem~\ref{thm:construction} are
satisfied, and~\eqref{eq:singleton-core} is exactly the corresponding
specialization of~\eqref{eq2}. The conclusion follows.
\end{proof}

When \(I_a^3=\{c\}\), the function \(H:I_a^3\times I_a^3\to X\) is completely determined by the single value \(h=H(c,c)\).
The following proposition characterizes
exactly which values of $h$ are admissible and, consequently, reduces
the existence, uniqueness, and enumeration of nullnorms with
fixed \(S\), \(T\), \(f\), and \(g\) to the study of an explicitly defined subset of $X$.

\begin{proposition}
\label{prop:singleton}
Let $a\in K\setminus\{0,1\}$ and suppose $I_a^3=\{c\}$.
Fix a t-conorm $S$ on $[0,a]$, a t-norm $T$ on $[a,1]$, and
order-preserving maps $f:X\to[a,1]$ and $g:X\to[0,a]$ satisfying
\textup{(C1)}.
Define $E_c$ to be the set of $h\in X$ satisfying
\eqref{eq:singleton-fibres}--\eqref{eq:singleton-upper}.
The nullnorms on $\mathbb P$ with absorbing element $a$ whose induced
data are the prescribed $S,T,f,g$ are in bijection with $E_c$, via
$h=H(c,c)$.
\end{proposition}
\begin{proof}
A function on the singleton $I_a^3\times I_a^3$ is automatically commutative and
increasing. If $h\in E_c$, define $H(c,c)=h$.
Equation~\eqref{eq:singleton-fibres} is precisely (C2), while
\eqref{eq:singleton-lower} and~\eqref{eq:singleton-upper} are precisely
(C3) and (C4). The sole instance of (C5) is
$\Lambda(h,c)=\Lambda(c,h)$, which follows from symmetry of the
piecewise definition. Theorem~\ref{thm:construction} therefore
produces a nullnorm with the prescribed \(S\), \(T\), \(f\) and \(g\).

Conversely, let a nullnorm have these $S,T,f,g$, and put
$h=H(c,c)$. The necessity part of
Theorem~\ref{thm:representation} gives (C2)--(C4), and hence
$h\in E_c$. The two assignments are inverse
because  \(H\) recovered from a nullnorm \(V\) is precisely the restriction \(V|_{I_a^3\times I_a^3}\).
\end{proof}

\begin{remark}
In the singleton case, the membership of the nullnorm determined by
$h\in E_c$ in the three subclasses is characterized as follows.
\begin{enumerate}
\item[(1)] The nullnorm $V$ belongs to $\mathcal V^{123}$ precisely when
$h\in I_a^1\cup I_a^2\cup\{c\}$.
\item[(2)] The nullnorm $V$ belongs to $\mathcal V_a$ precisely when
$h\notin I_a^1\cup I_a^2$.
\item[(3)] The nullnorm $V$ belongs to $\mathcal V_{cwa}$ precisely when
$h\in[0,a]\cup[a,1]$.
\end{enumerate}
\end{remark}




\subsection{The Subclasses
\texorpdfstring{$\mathcal V_a$ and $\mathcal V^{123}$}{Va and V123}}
\label{sec:subclasses}




We first reformulate Theorem~\ref{thm:construction} by partitioning
$\G$ according to the values attained by $H$ and  then restate
Theorem~4.1 of Xiu and Zheng~\cite{Xiu2025} using the notation adopted
in this manuscript. Together, these reformulations provide a unified
basis for the subsequent analysis of the subclasses $\mathcal{V}_a$
and $\mathcal{V}^{123}$.

\begin{theorem}
\label{thm:partitioned}
Let $\mathbb P$ be a bounded trellis and
$a\in X\setminus\{0,1\}$.
Let $S:[0,a]^2\to[0,a]$ be a t-conorm,
$T:[a,1]^2\to[a,1]$ a t-norm, $f:X\to[a,1]$ and
$g:X\to[0,a]$ order-preserving maps, and let $H:\G\to X$ be a
commutative and increasing function. Let
$\Lambda:\mathcal L_a\to X$ be the mixed interaction function given
by~\eqref{eq1}. Define
\[
\mathcal A_0=H^{-1}([a,1]),\qquad
\mathcal B_0=H^{-1}([0,a]\setminus\{a\}),
\]
\[
\mathcal D_i=H^{-1}(I_a^i)\quad(i=1,2,3),
\qquad
Q_i=H|_{\mathcal D_i}:\mathcal D_i\to I_a^i.
\]
Then
$
\G=\mathcal A_0\mathbin{\cup}\mathcal B_0
\mathbin{\cup}\mathcal D_1
\mathbin{\cup}\mathcal D_2
\mathbin{\cup}\mathcal D_3,
$
and the sets $\mathcal{A}_0$, $\mathcal{B}_0$  and  $\mathcal{D}_i$  are symmetric because $H$ is commutative.
Since each $Q_i$ is the restriction of $H$ to $\mathcal{D}_i$, the
commutativity and increasingness of $H$ imply that $Q_i$ is also
commutative and increasing.
Assume
that conditions~\textup{(C1)--(C5)} of
Theorem~\ref{thm:construction} hold. Then  $a\in K\setminus\{0,1\}$ if and only if the operation
$V:X^2\to X$ defined by
\begin{equation}\label{eq4}
V(x,y) =
\begin{cases}
S(g(x),g(y)), & (x,y)\in\M\cup\mathcal{B}_0,\\
T(f(x),f(y)), & (x,y)\in\N\cup\mathcal{A}_0,\\
Q_i(x,y), & (x,y)\in\mathcal{D}_i,\; i=1,2,3,
\end{cases}
\end{equation}
is a nullnorm on $\mathbb P$ with absorbing element $a$.

\end{theorem}

\begin{proof}
For $(x,y)\in\mathcal A_0$, conditions~\textup{(C1)} and
\textup{(C2)} give
$
H(x,y)=f(H(x,y))=T(f(x),f(y)).
$
For $(x,y)\in\mathcal B_0$, the same conditions give
$
H(x,y)=g(H(x,y))=S(g(x),g(y)).
$
On $\mathcal D_i$, we have $Q_i=H$. Consequently,
formula~\eqref{eq4} coincides with formula~\eqref{eq2}, and the
assertion follows from Theorem~\ref{thm:construction}.
\end{proof}



\begin{theorem}[\cite{Xiu2025}]
\label{thm:va-construction}
%
Let $\mathbb P=(X,\unlhd,\wedge,\vee,0,1)$ be a bounded
trellis and let $a\in X\setminus\{0,1\}$. Let
$S:[0,a]^2\to[0,a]$ be a t-conorm and
$T:[a,1]^2\to[a,1]$ a t-norm. Let
$\mathcal A,\mathcal B\subseteq\G$ be such that
$\M\cup\mathcal B$ is a symmetric lower subset of $X^2$ and
$\N\cup\mathcal A$ is a symmetric upper subset of $X^2$. Define
$
\mathcal D_3=\G\setminus(\mathcal A\cup\mathcal B)
$
and suppose that $\mathcal D_3$ is symmetric. Let
$f:X\to[a,1]$ and $g:X\to[0,a]$ be order-preserving maps, and let
$
H_1:\mathcal D_3\to I_a^3
$
be commutative and increasing.
Assume that for all $x,y,z,t\in X$, the following conditions hold:
\begin{itemize}
    \item[$(\mathrm{i})$] $(x,y)\in \mathcal{A} \Rightarrow S(g(x),g(y))=a$;

    \item[$(\mathrm{ii})$] $(x,y)\in \mathcal{B} \Rightarrow T(f(x),f(y))=a$;

    \item[$(\mathrm{iii})$] $f(x)=x$ for all $x\in[a,1]$, and
    $g(x)=x$ for all $x\in[0,a]$;

    \item[$(\mathrm{iv})$] For all $(x,y)\in\mathcal D_3$,
  $
    S(g(x),g(y))=g(H_1(x,y)),\
    T(f(x),f(y))=f(H_1(x,y));
   $

    \item[$(\mathrm{v})$] If $(x,z)\preccurlyeq(y,t)$,
    $(x,z)\in\M\cup\mathcal B$, and $(y,t)\in\mathcal D_3$, then
   $
    S(g(x),g(z))\unlhd H_1(y,t);
   $

    \item[$(\mathrm{vi})$] If $(x,z)\preccurlyeq(y,t)$,
    $(x,z)\in\mathcal D_3$, and $(y,t)\in\N\cup\mathcal A$, then
  $
    H_1(x,z)\unlhd T(f(y),f(t));
    $

    \item[$(\mathrm{vii})$] If $(x,y),(y,z)\in\mathcal D_3$, then
   $
    (H_1(x,y),z)\in\mathcal D_3
    \ \Leftrightarrow \
    (x,H_1(y,z))\in\mathcal D_3.
   $
    \item[$(\mathrm{viii})$] The function $H_1$ is conditionally
    associative: if
    $(x,y),(y,z),(H_1(x,y),z)\in\mathcal D_3$, then
    \[
    H_1(H_1(x,y), z) = H_1(x, H_1(y,z)).
    \]
\end{itemize}
Then $a\in K$ if and only if the operation
$V:X^2\to X$ defined by
\begin{equation}\label{eq6}
V(x,y) =
\begin{cases}
S(g(x),g(y)), & (x,y)\in\M\cup\mathcal{B}, \\
T(f(x),f(y)), & (x,y)\in\N\cup\mathcal{A}, \\
H_1(x,y), & (x,y)\in\mathcal{D}_3,
\end{cases}
\end{equation}
is a nullnorm on $\mathbb P$ with absorbing element $a$.
\end{theorem}

Theorem~\ref{thm:partitioned} uses the canonical regions determined by
the range of $H$, whereas Theorem~\ref{thm:va-construction} starts from
chosen branch domains $\mathcal A$ and $\mathcal B$, which may overlap.
The precise relationship between the two parametrizations is given next.

\begin{proposition}
\label{prop:degeneration}
At the level of the resulting operations, Theorem~\ref{thm:va-construction}
is exactly the $\mathcal V_a$ specialization of
Theorem~\ref{thm:partitioned}. More precisely, let $V$ be the operation
constructed in Theorem~\ref{thm:partitioned} and suppose that
$V\in\mathcal V_a$. Then
$
\mathcal D_1=\mathcal D_2=\varnothing.
$
With the choices
\[
\mathcal A=\mathcal A_0,\qquad
\mathcal B=\mathcal B_0,\qquad
H_1=Q_3=H|_{\mathcal D_3},
\]
the components satisfy all the hypotheses and conditions
\textup{(i)}--\textup{(viii)} of
Theorem~\ref{thm:va-construction}, and formulas~\eqref{eq4}
and~\eqref{eq6} define the same operation.

Conversely, if the components  in Theorem~\ref{thm:va-construction} yield a
nullnorm $V$, then $V\in\mathcal V_a$. The canonical function
$\widehat H=V|_{\G}$ and its range partition provide the components  satisfying
Theorem~\ref{thm:partitioned}, with
$\mathcal D_1=\mathcal D_2=\varnothing$, and again give the same
operation $V$.
\end{proposition}

\begin{proof}
The proof of Theorem~\ref{thm:partitioned} shows that
$V|_{\G}=H$. Hence, if $\mathcal D_i$ were nonempty for $i=1$ or
$i=2$, then $V$ would take a value in $I_a^i$, contrary to
$V\in\mathcal V_a$. Thus
$\mathcal D_1=\mathcal D_2=\varnothing$, and the canonical partition
reduces to
$
\G=\mathcal A_0\mathbin{\cup}\mathcal B_0
\mathbin{\cup}\mathcal D_3.
$

We next verify the order requirements in
Theorem~\ref{thm:va-construction}. The set
$\M\cup\mathcal B_0$ is lower. Indeed, let
$p\in\M\cup\mathcal B_0$ and $q\preccurlyeq p$. If $p\in\M$, then
$q\in\M$ because $\M$ is lower. Suppose that
$p\in\mathcal B_0$. If $q\notin\G$, then $q$ cannot belong only to
$\N$, since $\N$ is upper and this would imply $p\in\N$, contrary to
$p\in\G$; hence $q\in\M$. If $q\in\G$, monotonicity gives
\[
H(q)=V(q)\unlhd V(p)=H(p)\in[0,a]\setminus\{a\}.
\]
Lemma~\ref{lem:regional-sets123} and the defining range restriction of
$\mathcal V_a$ imply that $H(q)\in[0,a]$. The value $H(q)=a$ is
impossible by antisymmetry, and therefore $q\in\mathcal B_0$.
Consequently, $\M\cup\mathcal B_0$ is lower. The dual argument shows
that $\N\cup\mathcal A_0$ is upper. These sets are symmetric, and so is
$\mathcal D_3$.

Set $\mathcal A=\mathcal A_0$, $\mathcal B=\mathcal B_0$, and
$H_1=Q_3$. If $(x,y)\in\mathcal A$, then
\textup{(C2)} and~\eqref{eq:collapse} yield
$
S(g(x),g(y))=g(H(x,y))=a.
$
Similarly, if $(x,y)\in\mathcal B$, then
$T(f(x),f(y))=f(H(x,y))=a$. Hence conditions~\textup{(i)} and
\textup{(ii)} hold. Condition~\textup{(iii)} is \textup{(C1)}, while
condition~\textup{(iv)} is the restriction of \textup{(C2)} to
$\mathcal D_3$. Conditions~\textup{(v)} and \textup{(vi)} follow
directly from the increasingness of $V$ and its three branch formulas.

For conditions~\textup{(vii)} and \textup{(viii)}, let
$(x,y),(y,z)\in\mathcal D_3$. Then
$H_1(x,y),H_1(y,z)\in I_a^3$, so both outer pairs in
\[
V(H_1(x,y),z)=V(x,H_1(y,z))
\]
belong to $\G$. On $\G$, the value of $V$ belongs to $I_a^3$ exactly
when the relevant pair belongs to $\mathcal D_3$. Associativity of $V$
therefore gives
\[
(H_1(x,y),z)\in\mathcal D_3
 \Leftrightarrow \
(x,H_1(y,z))\in\mathcal D_3,
\]
which is \textup{(vii)}. When these equivalent memberships hold, the
same identity becomes
\[
H_1(H_1(x,y),z)=H_1(x,H_1(y,z)),
\]
which is \textup{(viii)}. Since
$\mathcal D_1=\mathcal D_2=\varnothing$, formulas~\eqref{eq4}
and~\eqref{eq6} coincide.

Conversely, suppose that Theorem~\ref{thm:va-construction} yields a
nullnorm $V$. Formula~\eqref{eq6} shows that
$
V(X^2)\cap(I_a^1\cup I_a^2)=\varnothing,
$
so $V\in\mathcal V_a$. Put $\widehat H=V|_{\G}$ and define its
canonical range partition as in Theorem~\ref{thm:partitioned}. Then
$\mathcal D_1=\mathcal D_2=\varnothing$. Moreover,
formula~\eqref{eq6} and condition~\textup{(iii)} recover
\[
S=V|_{[0,a]^2},\quad T=V|_{[a,1]^2},\quad
g(x)=V(x,0),\quad f(x)=V(x,1).
\]
The necessity part of Theorem~\ref{thm:representation}, applied with
$H=\widehat H$, now gives conditions~\textup{(C1)--(C5)}. Thus these components satisfy Theorem~\ref{thm:partitioned} and reproduce the
same operation $V$.
\end{proof}

\begin{remark}
The equivalence in Proposition~\ref{prop:degeneration} concerns the
resulting nullnorms, not the raw parameter sets. Theorem~\ref{thm:partitioned}
uses the unique partition induced by the range of $H$, whereas
Theorem~\ref{thm:va-construction} permits chosen branch domains
$\mathcal A$ and $\mathcal B$ that may overlap. Different such choices
can therefore determine the same nullnorm; passing to
$\widehat H=V|_{\G}$ gives the canonical normalization.
\end{remark}

By definition, $V\in\mathcal V^{123}$ implies
$V(I_a^3\times I_a^3)\subseteq I_a^1\cup I_a^2\cup I_a^3$.
Accordingly, in Theorem~\ref{thm:partitioned} we take
$H:\G\to I_a^1\cup I_a^2\cup I_a^3$, so that
$\mathcal A_0\cup\mathcal B_0=\varnothing$. This yields the following
construction theorem for $\mathcal V^{123}$.


\begin{theorem}
\label{thm:v123-construction}
Let $\mathbb P$ be a bounded trellis and
$a\in X\setminus\{0,1\}$.
Let $S:[0,a]^2\to[0,a]$ be a t-conorm,
$T:[a,1]^2\to[a,1]$ a t-norm, $f:X\to[a,1]$ and
$g:X\to[0,a]$ order-preserving maps, and
$H:\G\to I_a^1\cup I_a^2\cup I_a^3$ a commutative and increasing function.
Let $\Lambda:\mathcal L_a\to X$ be the mixed interaction function
given by~\eqref{eq1}.
For $i=1,2,3$, define
\[
\mathcal D_i=H^{-1}(I_a^i),\qquad
Q_i=H|_{\mathcal D_i}:\mathcal D_i\to I_a^i.
\]
Then $\G=\mathcal D_1\mathbin{\cup}\mathcal D_2
\mathbin{\cup}\mathcal D_3$, and
$Q_i$ are commutative and
increasing for $i=1,2,3$.
Assume that conditions~\textup{(C1)--(C5)} of
Theorem~\ref{thm:construction} hold.
Then  $a\in K$  if and only if  the operation
$V:X^2\to X$ defined by
\begin{equation}\label{eq9}
V(x,y)=
\begin{cases}
S(g(x),g(y)), & (x,y)\in\M,\\
T(f(x),f(y)), & (x,y)\in\N,\\
Q_i(x,y), & (x,y)\in\mathcal{D}_i,\; i=1,2,3,
\end{cases}
\end{equation}
is a nullnorm on $\mathbb P$ with absorbing element $a$, and
$V\in\mathcal V^{123}$.
\end{theorem}

\begin{proof}
Since $H=Q_i$ on $\mathcal D_i$, formula~\eqref{eq9} is precisely
formula~\eqref{eq2}. Conditions~\textup{(C1)--(C5)} are those of
Theorem~\ref{thm:construction}; hence $V$ is a nullnorm with absorbing
element $a$. Moreover,
$
V(I_a^3\times I_a^3)=H(\G)
\subseteq I_a^1\cup I_a^2\cup I_a^3,
$
so $V\in\mathcal V^{123}$.
\end{proof}

\begin{theorem}
\label{thm:v123-representation}
Let $\mathbb P=(X,\unlhd,\wedge,\vee,0,1)$ be a bounded trellis,
fix $a\in X\setminus\{0,1\}$, and let $V:X^2\to X$ be a binary
operation. Then the following statements are equivalent:
\begin{enumerate}
\item[$\mathrm{(1)}$] $V\in\mathcal{V}^{123}$.
\item[$\mathrm{(2)}$] $a\in K$ and
there exist \(S,T,f,g,H\)
satisfying the hypotheses and conditions of
Theorem~\ref{thm:v123-construction} such that $V$ is given
by~\eqref{eq9}.
\end{enumerate}
\end{theorem}

\begin{proof}
If $V\in\mathcal V^{123}$, apply
Theorem~\ref{thm:representation}. The recovered function
$H=V|_{\G}$ has range in $I_a^1\cup I_a^2\cup I_a^3$; its preimages
$\mathcal D_i=H^{-1}(I_a^i)$ and restrictions
$Q_i=H|_{\mathcal D_i}$ therefore give the components in
Theorem~\ref{thm:v123-construction}. Conversely, every operation
constructed in Theorem~\ref{thm:v123-construction} is a nullnorm and
its restriction to $I_a^3\times I_a^3$ has the required range. Hence
it belongs to $\mathcal V^{123}$.
\end{proof}


For any nullnorm $V$ represented by
Theorem~\ref{thm:representation}, the values of $V$ on
$X^2\setminus\G$ belong to $[0,a]\cup[a,1]$.
Consequently, membership in $\mathcal{V}_a$ is determined by the range
of $H$ on $\G$; more precisely,
\[
V\in\mathcal{V}_a
\quad\text{if and only if}\quad
H(\G)\subseteq[0,a]\cup[a,1]\cup I_a^3.
\]
Similarly,
\[
V\in\mathcal{V}_{cwa}
\quad\text{if and only if}\quad
H(\G)\subseteq[0,a]\cup[a,1].
\]

\begin{remark}
When $\mathbb{P}$ is a bounded lattice, the transitivity of its order
implies that $I_a^1=I_a^2=\varnothing$ and $I_a^3=I_a$.
Consequently, the range restriction in the definition of
$\mathcal{V}_a$ becomes vacuous, so that
$\mathcal{V}_a=\mathcal{V}$. Accordingly, our result recovers the
complete representation of nullnorms on bounded lattices established
in~\cite{Zhang2022}.
The subclasses $\mathcal V_{cwa}$ and $\mathcal V^{123}$ are
characterized, respectively, by
$
H(I_a\times I_a)\subseteq[0,a]\cup[a,1]
\ \text{and}\
H(I_a\times I_a)\subseteq I_a.
$
If $I_a\ne\varnothing$, then $I_a\times I_a$ is nonempty,
whereas
$
\bigl([0,a]\cup[a,1]\bigr)\cap I_a=\varnothing.
$
Hence no function on $I_a\times I_a$ can have its range contained in
both target sets, and therefore
$
\mathcal V_{cwa}\cap\mathcal V^{123}=\varnothing.
$
By contrast, if $I_a=\varnothing$, all the classes of nullnorms
considered above coincide.


\end{remark}

\section{Conclusion}\label{sec:conclusion}

This paper establishes a complete representation theorem for proper
nullnorms on bounded trellises without imposing an a priori range
restriction on $H:I_a^3\times I_a^3\to X$. For a fixed absorbing element \(a\),
the construction combines a t-conorm on $[0,a]$, a t-norm
on $[a,1]$, two order-preserving maps $f$ and $g$, and a
commutative, increasing function $H$.
The central device is the mixed
interaction function $\Lambda$, defined on pairs with at least one
component in $I_a^3$. It keeps the associativity identity well defined
even when an intermediate value of $H$ lies outside $I_a^3$ and thereby
allows the range of $H$ to remain unrestricted.

The special cases clarify how the representation depends on the size
of $I_a^3$. When $I_a^3=\varnothing$, no values of $H$ need to be
specified, and conditions~\textup{(C2)--(C5)} are vacuous. When
$I_a^3=\{c\}$, the admissible choices of $H(c,c)$ are described by the
set $E_c$, which parametrizes the resulting nullnorms for fixed $S$,
$T$, $f$, and $g$. The five-element lattice example shows that
condition~\textup{(C5)} does not follow from the component hypotheses
and conditions~\textup{(C1)--(C4)}. The fourteen-element nontransitive
example gives a nullnorm for which $H(I_a^3\times I_a^3)$ meets each
of $[0,a]$, $[a,1]$, $I_a^1$, $I_a^2$, and $I_a^3$, demonstrating the
need to admit the full range of $H$. The range-partition analysis also
clarifies the relation of the general framework to the subclasses
$\mathcal V_a$, $\mathcal V_{cwa}$, and $\mathcal V^{123}$ and to the
bounded-lattice case.

Several directions merit further investigation. First, intrinsic
criteria ensuring the existence of $S$, $T$, $f$, $g$  and $H$ for a prescribed
$a\in K$ would complement the present representation, since membership
in $K$ alone does not construct these components.
Second, the
parametrization by $E_c$ suggests extending the analysis to finite sets
$I_a^3$ and developing algorithms for deciding existence and
enumerating nullnorms on finite bounded trellises. Third, additional
properties, such as idempotency and distributivity, could be
characterized directly in terms of $S$, $T$, $f$, $g$, $H$  and
$\Lambda$.

\def\bibsection{\section*{\refname}}

\end{document}